\documentclass[uplatex,12pt]{article}
\usepackage{amsmath,amssymb,amsthm}

\newtheorem{definition}{Definition}%[section]
\newtheorem{theorem}[definition]{Theorem}

\newtheorem*{definition*}{Definition}
\newtheorem*{theorem*}{Theorem}
\newtheorem*{proposition*}{Proposition}
\newtheorem*{example*}{Example}
\newtheorem*{exercise*}{Exercise}
\newtheorem*{corollary*}{Corollary}
\newtheorem*{remark*}{Remark}

\usepackage{geometry}
\usepackage[shortlabels, inline]{enumitem}
\usepackage{bm}
\usepackage[dvipdfmx]{graphicx}
\usepackage{float}
\allowdisplaybreaks[2]
\usepackage[dvipdfmx]{color}
\usepackage{mathtools}
\mathtoolsset{showonlyrefs=true}
\usepackage{url}

\begin{document}

\begin{center}
  ~\vspace{20pt}
  
  \Large 
  Poisson Hyperbolic Staircase in Discrete Time

  \vspace{20pt}
  
  \large
  Naohiro Yoshida

  \normalsize
  Department of Economics,
  Keiai University

  1-5-21,
  Anagawa, Inage,
  263-8588,
  Chiba,
  Japan

  E-mail:
  \url{n-yoshida@u-keiai.ac.jp}

\end{center}

  \vspace{20pt}

\noindent
  MSC2020:
  60J05    ;
  60E10

\noindent
  Keywords:
  Poisson hyperbolic staircase    ;
  Discrete-time Markov process    ;
  State-dependent transition    ;
  Counting process    ;
  Probability generating function

\begin{abstract}
In this paper, we propose a novel stochastic process that serves as a natural discrete-time counterpart to the continuous-time model known as the ``Poisson hyperbolic staircase'' proposed by Levikson et al. (1999), and clarify its analytical properties. The proposed model is a Markov chain on the state space $(0,1]$. Its transition rule states that at each time step, it jumps downwards to a value less than or equal to the current state according to a continuous uniform distribution with a probability proportional to the current state, and otherwise remains in the same state. In the analysis of the continuous-time model, the scaling property based on the continuity of time and space serves as a powerful tool. However, for this discrete-time process, an essential analytical difficulty arises because this scaling property is inapplicable. To overcome this difficulty, we adopt an approach that directly evaluates recurrence relations and integral equations. First, starting from the conditional transition of this process, we derive closed-form expressions for the marginal distribution and the joint survival function. Next, focusing on the counting process representing the number of jump occurrences and the sum of the state variables, we provide exact closed-form expressions for the probability generating function and the Laplace transform. Furthermore, we clarify the necessary and sufficient conditions that a sequence of functions must satisfy to construct a martingale associated with this process, and present a concrete sequence of martingales.
\end{abstract}

\newpage
\section{Introduction}

Stochastic processes with state-dependent transition probabilities (or transition rates) play a crucial role in various fields of applied probability, such as modeling system failures in reliability engineering, extreme value theory, and the mathematics of record values. For instance, \cite{barlow1975statistical} discussed processes with various transition characteristics in the context of reliability theory, and extensive research on the statistical properties of record values has been conducted by \cite{arnold1998records}.

In this context, an interesting object among continuous-time models is the ``Poisson hyperbolic staircase'' proposed by \cite{levikson1999poisson}. This process behaves such that when it is at a certain height $y$, it remains in that state for a time distributed according to an exponential distribution with parameter $y$, and then jumps (descends) to a new height according to a uniform distribution on the interval $(0, y)$. This structure also has deep connections with the extremal processes discussed by \cite{dwass1964extremal}.

The purpose of this paper is to construct a natural ``discrete-time model'' of this Poisson hyperbolic staircase and strictly clarify its stochastic properties.
Specifically, assuming $0 < p \leq 1$ and starting from the initial value $X_0=1$, we consider a Markov chain $\{X_n\}_{n=0}^\infty$ where, given the state $X_{n-1}$ at time $n-1$, it transitions to a value according to a continuous uniform distribution on the interval $(0, X_{n-1})$ with probability $p X_{n-1}$, and remains in the original state $X_{n-1}$ with probability $1 - p X_{n-1}$. The time evolution of the model is described as follows:
\begin{align}
    X_0=1,\quad X_n = 
    \begin{cases}
        X_{n-1} Z_n & \mbox{with probability } p X_{n-1},  \\
        X_{n-1} & \mbox{with probability } 1 - p X_{n-1},
    \end{cases}
    \quad n=1,2,\dots ,
\end{align}
where $Z_n \sim U(0,1)$ for $n=1,2,\dots$ are independent and identically distributed random variables.
The property of a ``jump rate proportional to height'' in the continuous-time model is translated to a ``jump probability proportional to height'' in the discrete-time model.

However, there is an essential difference in the analytical methods compared to the continuous-time model. In the continuous-time model, a ``scaling property'' due to the continuity of time and space holds, which serves as a powerful weapon to elegantly derive various expectations and distributions. On the other hand, in this discretized model, it is not possible to continuously scale time, meaning the scaling property cannot be applied directly, leading to analytical difficulties. Therefore, in this paper, without relying on the scaling property, we adopt a direct analytical approach using integral equations and generating functions.

The structure of this paper is as follows. In Section 2, we derive the marginal distribution and the joint survival function of the process $\{X_n\}$. In Section 3, we treat the total number of jumps $N_n$ as a counting process and provide a closed-form expression for its probability generating function. Finally, in Section 4, we clarify the necessary and sufficient conditions for constructing martingales associated with this process.

\section{Probability distributions}

Let $\xi_i,~i=1,2,\dots$ be the indicator random variable representing whether a jump occurs at time $i$, i.e., $\xi_i = 1$ if a jump occurs (or if $X_i < X_{i-1}$) and $\xi_i = 0$ otherwise (or if $X_i = X_{i-1}$).
First, we investigate the marginal distribution of $\{X_n\}$.
The marginal distribution can be obtained by conditioning on the state at the previous time step.

\begin{theorem}
For any positive integer $n$ and $0<x\leq 1$,
\begin{align}
  &P(X_n=1)=(1-p)^n, 
  \\
  &P(X_n\leq x)=1-(1-px)^n, \quad \text{for } 0<x<1.
\end{align}
\end{theorem}

\begin{proof}
First, note that the following conditional distribution is obtained:
\begin{align}
  P(X_n\leq x|X_{n-1})=\begin{cases}
    1, & \text{for } x\geq X_{n-1},
    \\
    px, & \text{for } 0< x< X_{n-1},
    \\
    0, & \text{for } x\leq 0.
  \end{cases}
  \label{eq:condi}
\end{align}
Indeed, for $0<x<X_{n-1}$,
\begin{align}
  P(X_n\leq x|X_{n-1})&=P(\xi_n=1 \cap Z_{n}X_{n-1}\leq x|X_{n-1})
  \\
  &=P(\xi_n=1|X_{n-1})P\left(Z_n\leq \frac{x}{X_{n-1}}\middle|X_{n-1}\right)
  \\
  &=pX_{n-1}\cdot \frac{x}{X_{n-1}}
  \\
  &=px.
\end{align}
Using this conditional distribution, the unconditional distribution can be calculated as follows.
For $0<x<1$,
\begin{align}
  P(X_n\leq x)&=E[P(X_n\leq x|X_{n-1})]
  \\
  &=1\cdot P(X_{n-1}\leq x)+pxP(X_{n-1}>x)
  \\
  &=P(X_{n-1}\leq x)+px(1-P(X_{n-1}\leq x))
  \\
  &=px+(1-px)P(X_{n-1}\leq x).
\end{align}
Let $a_n=P(X_n\leq x)$.
Noting that $a_0=P(X_0\leq x)=P(1\leq x)=0$, solving the recurrence relation
$
  a_n=px+(1-px)a_{n-1}
$
yields
\begin{align}
  a_n&=px+(1-px)px+(1-px)^2px+\dots+(1-px)^{n-1}px
  \\
  &=\frac{px-(1-px)^npx}{1-(1-px)}
  % \\
  % &=\frac{px-(1-px)^npx}{px}
  \\
  &=1-(1-px)^n.
\end{align}
Also, for $x=1$, it is clear that $P(X_n=1)=(1-p)^n$.
The assertion follows from the above.
\end{proof}

\begin{theorem}
For any positive integer $n$ and $0\leq x_i<1~(i=1,2,\dots,n)$,
\begin{align}
  P(X_1>x_1\cap X_2>x_2\cap \dots\cap X_n>x_n)
  =\prod_{i=1}^n \left(1-p\cdot \max(x_{n-i+1},x_{n-i+2},\dots,x_n) \right).
\end{align}
\end{theorem}

\begin{proof}
Let $I(A)$ denote the indicator function of an event $A$; that is, $I(A)=1$ if $A$ occurs, and $I(A)=0$ if $A$ does not occur.
From \eqref{eq:condi}, note that
$$E[I(X_n>x_n)|X_1,\dots,X_{n-1}]
 =I(X_{n-1}>x_n)P(X_n>x_n|X_{n-1})=I(X_{n-1}>x_n)(1-px_n).$$
Then, by the tower property of conditional expectation, we obtain the relationship:
\begin{align}
  &P(X_1>x_1\cap X_2>x_2\cap \dots\cap X_n>x_n)
  \\
  &=E[I(X_1>x_1\cap \dots\cap X_{n-1}>x_{n-1})E[I(X_n>x_n)|X_1,\dots,X_{n-1}]]
 \\
 &=E[I(X_1>x_1\cap \dots\cap X_{n-1}>x_{n-1})I(X_{n-1}>x_n)](1-px_n)
 \\
 &=P(X_1>x_1\cap\dots\cap X_{n-2}>x_{n-2}\cap X_{n-1}>\max\{x_{n-1},x_n\})(1-px_n).
\end{align}
Repeating this process yields the assertion of the theorem.
\end{proof}

\section{Counting process}

We define the counting process as
$N_0=0,~N_n=\xi_1+\xi_2+\dots+\xi_n~(n\geq 1)$.
The probability distribution of $N_n$ can be obtained by finding the generating function of the probability generating functions.
For the general mathematics of such counting processes and the basics of analytical methods using generating functions and Laplace transforms, refer to standard texts such as Daley and Vere-Jones (1988) \cite{daley1988introduction}.

\begin{theorem}
For any positive integer $n$ and $k=0,1,2,\dots,n$, the probability generating function of $N_n$ for $-1\leq z\leq 1$ is
\begin{align}
  E[z^{N_n}]
  =\sum_{j=0}^n \binom{n}{j} \binom{z-1}{j} p^j 
  = \sum_{k=0}^n \left( \sum_{j=k}^n \binom{n}{j} p^j c_{j,k} \right) z^k
  \label{eq:pgf}
\end{align}
where $c_{j,k}$ is the coefficient of $z^k$ in the expansion of $\binom{z-1}{j}$, i.e., $\binom{z-1}{j} = \sum_{k=0}^j c_{j,k} z^k$.
Therefore, the probability mass function $P(N_n=k),~k=0,1,2,\dots,n$ is given by
\begin{align}
    P(N_n=k) = \sum_{j=k}^n \binom{n}{j} p^j c_{j,k}.
\end{align}
\end{theorem}

\begin{proof}
Conditioning on the realization at time $n=1$, for $0<x\leq 1$ we obtain the relationship:
  \begin{align}
    E[z^{N_n}|X_0=x]
    &=pxE[z^{N_n}|\xi_1=1,X_0=x]+(1-px)E[z^{N_n}|\xi_1=0,X_0=x]
    \\
    &=px\int_0^1E[z^{N_{n}}|\xi_1=1,Z_1=t,X_0=x]dt+(1-px)E[z^{N_{n-1}}|X_0=x]
    \\
    &=pxz\int_0^1E[z^{N_{n-1}}|X_0=xt]dt+(1-px)E[z^{N_{n-1}}|X_0=x].
  \end{align}
Therefore, defining $G_n(x)=E[z^{N_n}|X_0=x]$, $G_n(x)$ satisfies the following integral equation:
\begin{align}
    G_n(x) &= p x z \int_0^1 G_{n-1}(xt) \, dt + (1-px) G_{n-1}(x) \label{eq:recurrence}
\end{align}
with the initial condition $G_0(x) = 1,~\forall x$.
Substituting $u = xt$ in the integral part yields
\begin{align}
    G_n(x) = p z \int_0^x G_{n-1}(u) \, du + (1-px) G_{n-1}(x).
\end{align}
Here, we introduce the bivariate generating function $F(x, y) = \sum_{n=0}^\infty G_n(x) y^n$ for the sequence $\{G_n(x)\}_{n=0}^\infty$. Multiplying both sides of the recurrence relation by $y^n$ and summing from $n=1$ to infinity gives
\begin{align}
    F(x, y) - 1 = p z y \int_0^x F(u, y) \, du + y(1-px) F(x, y).
\end{align}
Rearranging this, we obtain the following integral equation:
\begin{align}
    (1 - y + pxy) F(x, y) = 1 + p z y \int_0^x F(u, y) \, du.
\end{align}
Differentiating both sides with respect to $x$ yields
\begin{align}
    p y F(x, y) + (1 - y + pxy) \frac{\partial F(x, y)}{\partial x} = p z y F(x, y)
\end{align}
\begin{align}
    (1 - y + pxy) \frac{\partial F(x, y)}{\partial x} = p y (z-1) F(x, y).
\end{align}
Thus, we obtain a separable first-order differential equation. Solving this yields
\begin{align}
    \frac{1}{F(x, y)} \frac{\partial F(x, y)}{\partial x} &= \frac{py(z-1)}{1 - y + pxy} \\
    \ln F(x, y) &= (z-1) \ln(1 - y + pxy) + C(y) \\
    F(x, y) &= C_1(y) (1 - y + pxy)^{z-1}.
\end{align}
From the original integral equation, when $x=0$, we have $(1-y)F(0, y) = 1$, i.e., $F(0, y) = (1-y)^{-1}$. Hence,
$
    C_1(y) (1-y)^{z-1} = (1-y)^{-1}, 
$
which means $ C_1(y) = (1-y)^{-z} $.
Therefore, the closed form of the generating function is as follows:
\begin{align}
    F(x, y) &= (1-y)^{-z} (1 - y + pxy)^{z-1} \\
    &= \frac{1}{1-y} \left( 1 + \frac{pxy}{1-y} \right)^{z-1}.
\end{align}
We expand this with respect to $y$ and extract the coefficient of $y^n$. First, by the negative binomial theorem,
\begin{align}
    F(x, y) &= \frac{1}{1-y} \sum_{k=0}^\infty \binom{z-1}{k} \left( \frac{pxy}{1-y} \right)^k \\
    &= \sum_{k=0}^\infty \binom{z-1}{k} p^k x^k \frac{y^k}{(1-y)^{k+1}}.
\end{align}
Furthermore, since $\frac{y^k}{(1-y)^{k+1}} =\sum_{m=0}^\infty \binom{m+k}{k}y^{m+k}= \sum_{n=k}^\infty \binom{n}{k} y^n$ by the negative binomial theorem,
\begin{align}
    F(x, y) &= \sum_{k=0}^\infty \binom{z-1}{k} p^k x^k \sum_{n=k}^\infty \binom{n}{k} y^n \\
    &= \sum_{n=0}^\infty \left( \sum_{k=0}^n \binom{n}{k} \binom{z-1}{k} p^k x^k \right) y^n.
\end{align}
By definition, since $G_n(x)$ is the coefficient of $y^n$, we obtain
$$G_n(x)=\sum_{k=0}^n \binom{n}{k} \binom{z-1}{k} p^k x^k.$$
Since $E[z^{N_n}]=E[z^{N_n}|X_0=1]=G_n(1)$, we find
$$E[z^{N_n}]=\sum_{j=0}^n \binom{n}{j} \binom{z-1}{j} p^j.$$
Furthermore, substituting $\binom{z-1}{j} = \sum_{k=0}^j c_{j,k} z^k$ into this equation gives
\begin{align}
    E[z^{N_n}] &= \sum_{j=0}^n \binom{n}{j} p^j \sum_{k=0}^j c_{j,k} z^k \nonumber \\
    &= \sum_{k=0}^n \left( \sum_{j=k}^n \binom{n}{j} p^j c_{j,k} \right) z^k
\end{align}
yielding the desired result.
\end{proof}

Consider the sum of the process up to time $n$, i.e., $\sum_{i=1}^n X_i$.
To investigate its probability generating function, we determine the generating function of the probability generating functions.

\begin{theorem}
For $t\geq 0$, the generating function of
$
  E[e^{-t\sum_{i=1}^n X_i}]
$
is given by
\begin{align}
    \sum_{n=0}^\infty E[e^{-t\sum_{i=1}^n X_i}]z^n 
    = \frac{1}{1 - z(1-p)e^{-t}} \exp\left( \int_0^1 \frac{p z e^{-ts}}{1 - z(1-ps)e^{-ts}} \, ds \right).
\end{align}
\end{theorem}
\begin{proof}
  Let
  \begin{align}
  W_n(x) = E\left[e^{-t \sum_{i=1}^n X_i} \middle| X_0 = x\right].
\end{align}
Conditioning on the realization of the first jump, we have
\begin{align}
  W_n(x) &= E\left[ e^{-t X_1} W_{n-1}(X_1) \middle| X_0 = x \right]
  \\
  &= (1-px) e^{-tx} W_{n-1}(x) + px \int_0^1 e^{-t(ux)} W_{n-1}(ux) du
  \\
  &= (1-px) e^{-tx} W_{n-1}(x) + p \int_0^x e^{-ty} W_{n-1}(y) dy.
  \label{eq:W_recurrence}
\end{align}
Let $H(x, z) = \sum_{n=0}^\infty W_n(x) z^n$ be the generating function for the sequence $\{W_n(x)\}_{n=0}^\infty$.
Multiplying both sides of Eq. \eqref{eq:W_recurrence} by $z^n$ and summing from $n=1$ to $\infty$, the left-hand side becomes $H(x, z) - W_0(x) = H(x, z) - 1$, yielding
\begin{align}
    H(x, z) - 1 = z(1-px)e^{-tx} H(x, z) + p z \int_0^x e^{-ty} H(y, z) \, dy.
\end{align}
Rearranging this gives
\begin{align}
    \left( 1 - z(1-px)e^{-tx} \right) H(x, z) = 1 + p z \int_0^x e^{-ty} H(y, z) \, dy. \label{eq:integral_H}
\end{align}
Let $D(x) = 1 - z(1-px)e^{-tx}$. Differentiating both sides of Eq. \eqref{eq:integral_H} with respect to $x$ yields
\begin{align}
    D'(x) H(x, z) + D(x) \frac{\partial H(x, z)}{\partial x} = p z e^{-tx} H(x, z).
\end{align}
We transform this into a separable form:
\begin{align}
    \frac{1}{H(x, z)} \frac{\partial H(x, z)}{\partial x} &= \frac{p z e^{-tx} - D'(x)}{D(x)} 
    \\
    &= \frac{p z e^{-tx}}{1 - z(1-px)e^{-tx}} - \frac{D'(x)}{D(x)}.
\end{align}
Noting from Eq. \eqref{eq:integral_H} that the initial value is $H(0, z) = \frac{1}{1-z}$ and that $D(0) = 1-z$, we obtain
\begin{align}
    \ln H(x, z) - \ln\left(\frac{1}{1-z}\right) &= \int_0^x \frac{p z e^{-ts}}{1 - z(1-ps)e^{-ts}} \, ds - \left( \ln D(x) - \ln(1-z) \right) \nonumber \\
    \ln H(x, z) &= - \ln D(x) + \int_0^x \frac{p z e^{-ts}}{1 - z(1-ps)e^{-ts}} \, ds.
\end{align}
Thus, we find
\begin{align}
    H(x, z) = \frac{1}{1 - z(1-px)e^{-tx}} \exp\left( \int_0^x \frac{p z e^{-ts}}{1 - z(1-ps)e^{-ts}} \, ds \right).
\end{align}
The assertion of the theorem corresponds to the case when $x=1$.
\end{proof}

We define
\begin{align}
    c_k(x) = p \int_0^x e^{-tks} (1-ps)^{k-1} \, ds \quad (k \geq 1)
\end{align}
and determine $V_i(x)$ by the power series $\exp\left( \sum_{k=1}^\infty c_k(x) z^k \right) = \sum_{i=0}^\infty V_i(x) z^i$ (with $V_0(x)=1$). Then, $W_n(x)$ is given by the following equation:
\begin{align}
    W_n(x) = \sum_{i=0}^n (1-px)^{n-i} e^{-t(n-i)x} V_i(x).
\end{align}
Indeed, this can be verified as follows.
We expand the generating function $H(x, z)$ by separating it into a fractional part and an exponential part.
For the fractional part, from the expansion of a geometric series,
\begin{align}
    \frac{1}{1 - z(1-px)e^{-tx}} = \sum_{j=0}^\infty z^j (1-px)^j e^{-tjx}.
\end{align}
For the integral term inside the exponential function, expanding the denominator of the integrand and performing term-by-term integration yields
\begin{align}
    \int_0^x \frac{p z e^{-ts}}{1 - z(1-ps)e^{-ts}} \, ds &= \int_0^x \sum_{j=0}^\infty p z^{j+1} e^{-t(j+1)s} (1-ps)^j \, ds \nonumber \\
    &= \sum_{k=1}^\infty \left( p \int_0^x e^{-tks} (1-ps)^{k-1} \, ds \right) z^k = \sum_{k=1}^\infty c_k(x) z^k.
\end{align}
Therefore, the generating function can be written as the product of two series as follows:
\begin{align}
    H(x, z) = \left( \sum_{j=0}^\infty z^j (1-px)^j e^{-tjx} \right) \left( \sum_{i=0}^\infty V_i(x) z^i \right).
\end{align}
By comparing the coefficients of $z^n$ on the right-hand side (Cauchy product), we obtain the assertion.

\section{Martingale}

In this section, we clarify the conditions for constructing a martingale associated with the proposed discrete-time process $\{X_n\}$. For the basic theory of discrete-time martingales, refer to \cite{williams1991probability}.

\begin{theorem}
Let $f_n(x)$ for $n=0,1,2,\dots$ be functions of class $C^1$ defined on $0<x\leq 1$.
Then, a necessary and sufficient condition for $\{f_n(X_n)\}_{n=0}^\infty$ to be a martingale with respect to $\{X_n\}_{n=0}^\infty$ is that
\begin{equation}
  f_n(x) = f_0(0) + \int_0^x \frac{f_0'(t)}{(1-pt)^n} \, dt \quad (n=0,1,2,\dots).
  \label{eq:martingale_cond}
\end{equation}
\end{theorem}
\begin{proof}
Since
\begin{align}
   E[f_n(X_n) \mid X_{n-1} = x] 
   &=(1-px)f_n(x)+px\int_0^1 f_n(xu)du
   \\
   &=(1-px)f_n(x)+p\int_0^x f_n(y)dy,
 \end{align}
note that $E[f_n(X_n) \mid X_{n-1} = x] = f_{n-1}(x)$ holds if and only if
\begin{align}
    (1-px) f_n(x) + p \int_0^x f_n(y) \, dy = f_{n-1}(x) \label{eq:martingale_def}
\end{align}
holds.

First, we prove sufficiency.
Assume $f_n(x)$ is defined by Eq. \eqref{eq:martingale_cond}. Then,
\begin{align}
  (1-px) f_n(x) + p \int_0^x f_n(y) \, dy 
  &=  (1-px) f_n(x) + p \left( x f_n(x) - \int_0^x y f_n'(y) \, dy \right)  \\
    &= f_n(x) - \int_0^x p y f'_n(y)\,dy.
\end{align}
Here, substituting \eqref{eq:martingale_cond} rewritten as $f_n(x) = f_0(0) + \int_0^x f_n'(y) \, dy$ into the above equation gives
\begin{align}
    f_n(x) - \int_0^x p y f'_n(y)\,dy&= f_0(0) + \int_0^x f_n'(y) \, dy - \int_0^x p y f_n'(y) \, dy \nonumber \\
    &= f_0(0) + \int_0^x (1-py) f_n'(y) \, dy \nonumber \\
    &= f_0(0) + \int_0^x (1-py) \frac{f_0'(y)}{(1-py)^n} \, dy \nonumber \\
    &= f_0(0) + \int_0^x \frac{f_0'(y)}{(1-py)^{n-1}} \, dy.
\end{align}
This final expression is precisely the definition of $f_{n-1}(x)$. Therefore, it is shown that $\{f_n(X_n)\}_{n=0}^\infty$ is a martingale.

Next, we prove necessity.
Assume $\{f_n(X_n)\}_{n=0}^\infty$ is a martingale, that is, Eq. \eqref{eq:martingale_def} holds.
Since $f_n(x)$ is differentiable, differentiating both sides of Eq. \eqref{eq:martingale_def} with respect to $x$ gives
\begin{align}
    -p f_n(x) + (1-px) f_n'(x) + p f_n(x) &= f_{n-1}'(x) \nonumber \\
    (1-px) f_n'(x) &= f_{n-1}'(x) \nonumber \\
    f_n'(x) &= \frac{f_{n-1}'(x)}{1-px}.
\end{align}
By repeatedly applying this recurrence relation, we obtain the following equation:
\begin{align}
    f_n'(x) = \frac{f_0'(x)}{(1-px)^n}. \label{eq:derivative_rel}
\end{align}
Also, substituting $x=0$ into Eq. \eqref{eq:martingale_def} gives
$
    (1-0) f_n(0) + 0 = f_{n-1}(0),
$
that is, $ f_n(0) = f_{n-1}(0), $
which shows that $f_n(0) = f_0(0)$ for all $n$.
Finally, integrating Eq. \eqref{eq:derivative_rel} from $t=0$ to $x$ yields
\begin{align}
    \int_0^x f_n'(t) \, dt &= \int_0^x \frac{f_0'(t)}{(1-pt)^n} \, dt  \\
    % f_n(x) - f_n(0) &= \int_0^x \frac{f_0'(t)}{(1-pt)^n} \, dt  \\
    f_n(x) &= f_0(0) + \int_0^x \frac{f_0'(t)}{(1-pt)^n} \, dt.
\end{align}
From the above, it has been shown that $f_n(x)$ must be of this form to be a martingale. Combined with sufficiency, the proof is completed.
\end{proof}

For example, when
\begin{align}
    f_0(x) = -\frac{1}{p} \ln(1-px),
\end{align}
the initial value at $x=0$ is $f_0(0) = -\frac{1}{p} \ln(1) = 0$.
Also,
\begin{align}
    f_0'(x) = -\frac{1}{p} \cdot \frac{-p}{1-px} = \frac{1}{1-px}.
\end{align}
Therefore, for any $n \geq 1$,
\begin{align}
    f_n(x) = 0 + \int_0^x \frac{\frac{1}{1-pt}}{(1-pt)^n} \, dt 
    % = \int_0^x \frac{1}{(1-pt)^{n+1}} \, dt 
    = \int_0^x (1-pt)^{-(n+1)} \, dt.
\end{align}
Thus,
\begin{align}
    f_n(x) = \frac{1}{pn} \left[ (1-pt)^{-n} \right]_0^x 
    = \frac{1}{pn} \left( \frac{1}{(1-px)^n} - 1 \right),
\end{align}
which shows that
$$ f_0(X_0)= -\frac{1}{p} \ln(1-pX_0),\quad \{f_n(X_n)\}_{n=1}^\infty=\left\{\frac{1}{pn} \left( \frac{1}{(1-pX_n)^n} - 1 \right)\right\}_{n=1}^\infty$$
is a martingale.

\bibliography{ref}
\bibliographystyle{apalike} 

\end{document}